\documentclass[10pt,journal,twoside,web]{article}

\usepackage{cite}
\usepackage{amsmath,amssymb,amsfonts,amsthm}
\usepackage{mathtools}
\usepackage{algorithmic}
\usepackage{graphicx}
\usepackage{textcomp}
\usepackage{ragged2e}
\usepackage{tikz-network}
\usepackage{hyperref}
\usepackage{csquotes}
\usepackage{moresize}
\newtheorem*{keywords}{Keywords}
\newtheorem{thm}{Theorem}[section]
\newtheorem{eg}{Example}[section]
\newtheorem{rem}{Remark}[section]
\newtheorem{cor}{Corollary}[section]
\newtheorem{Nt}{Note}[section]
\newtheorem{Definition}{Definition}[section]

\begin{document}
	\title{Controllability and Observability of Heterogeneous Networked Systems with Non-Uniform Node Dimensions and Distinct Inner-Coupling Matrices}
	
	\author{Aleena Thomas, Abhijith Ajayakumar, and Raju K.George
		\thanks{The first and second authors are funded by the University Grants Commission, India and Council for Scientific and Industrial Research, India,  respectively.}
		\thanks{The authors are with the Department of Mathematics, Indian Institute of Space Science and Technology, Thiruvananthapuram, Kerala, India.\url{aleenathomas.22@res.iist.ac.in, abhijithajayakumar.19@res.iist.ac.in, george@iist.ac.in}}}

	\maketitle
	
	\begin{abstract}
		In this paper we extend the work in \cite{A.Thomas} wherein the controllability and observability of a heterogeneous networked system with distinct node dimensions were studied. This paper adds to \cite{A.Thomas} a necessary and sufficient condition for controllability of the networked system. The result demonstrates the dependence of controllability of the network on factors like network topology, inner interactions among nodes and nodal dynamics. The result is formulated by characterizing the left eigenvectors of the network state matrix. Another necessary and sufficient condition for controllability, which is a reformulation of the \textit{Popov-Belevitch-Hautus} controllability condition, a necessary and sufficient condition for observability of the networked system and certain necessary conditions for controllability of the networked system are the other results established in this paper. Variants of these results under certain specific network topologies like path, cycle, star and wheel are also discussed.
	\end{abstract}
	
	\begin{keywords}
		Controllability, Inner-Coupling matrices, Node dimensions, Heterogeneous dynamics, Linear Time Invariant systems, Networked Systems, Observability.
	\end{keywords}
	
	\section{Introduction}
	\label{sec:introduction}
	The study of Mathematical Theory of Control was initiated by R. E. Kalman in early 1960's (refer \cite{R.Kalman} and \cite{R.Kalman-2}). He characterized conditions for controllability and observability of stand-alone systems. Subsequently, control scientists have derived numerous equivalent conditions for these notions, the most remarkable of which are \textit{Popov-Belevitch-Hautus (PBH) eigenvector tests} and \textit{PBH rank tests}  for controllability and observability. Owing to the complexity of real-world systems, the notion of networked systems was introduced, which found applications in various fields like Gene Regulatory Networks \cite{M.Levine}, Power Grids \cite{H.Farhangi}, brain neuronal connections, the World Wide Web etc.\cite{D.S.Bassett}.\\
	
	The study of Linear Time Invariant (LTI),  Multi Input Multi Output (MIMO) networked systems was initiated by Wang \textit{et al.}\cite{L.Wang}. Networked systems in which all nodes follow the same dynamics are called homogeneous networked systems wheras, those with distinct dynamics are called heterogeneous networked systems. In Wang \textit{et al.} \cite{L.Wang}, homogeneous networked systems are considered; a necessary and sufficient condition for the controllability of the system, as well as some necessary conditions for controllability are presented. Further, Wang \textit{et al.}\cite{L.Wang-2} investigated the controllability of a homogeneous networked system, where communications are performed through one-dimensional channels. Wang \textit{et al.}'s\cite{L.Wang} necessary and sufficient condition is a reformulation of the PBH eigenvector test for controllability; verification of which is computationally demanding as it involves solving a system of matrix equations. Xiang \textit{et al.} \cite{L.Xiang} extended the work of Wang \textit{et al.}\cite{L.Wang} to networked systems with heterogeneous nodal dynamics and with the same inner-coupling matrices. Controllability condition in \cite{L.Xiang} does not hold true in general. Ajayakumar and George constructed a counter-example to establish this and proved a new result.  Hao \textit{et al.}'s\cite{Y.Hao} necessary and sufficient condition overcame the computational difficulty in Wang \textit{et al.}'s result for homogeneous networked systems wherein  individual nodes are connected in a diagonalizable network topology. This condition gives more information in the sense that it illustrates the dependence of the controllability of networked system on factors like individual node dynamics, inner interactions among nodes, and network topology. The result was formulated by characterizing the left eigenvectors of the network state matrix. Hao \textit{et al.}'s result was extended to heterogeneous networked systems under triangularizable network topology by Ajayakumar and George\cite{A.Ajayakumar-2}, \cite{A.Ajayakumar-3}. Controllability of heterogeneous networked systems with different inner-coupling matrices were studied, and a necessary and sufficient  condition for controllability was established by Kong \textit{et al.}\cite{Z.Kong}. In all these works, the dimensions of the individual nodes in the networked system are the same. We established conditions for controllability and observability for heterogeneous networked systems with distinct node dimensions in our previous work \cite{A.Thomas}. Apart from extending Wang \textit{et al.}'s\cite{L.Wang} results, in \cite{A.Thomas} we have obtained a necessary and sufficient  condition for the observability of the networked system. A survey of controllability of LTI networked systems can be found in \cite{A.Ajayakumar-4} and \cite{L.Xiang}.\\
	
	This paper extends the results in \cite{A.Thomas} by establishing another necessary and sufficient condition for the controllability of the networked system. This condition is more efficient as compared to the necessary and sufficient  condition obtained as a reformulation of the PBH condition, as it shows the interplay between    the controllability of the network and factors like network topology, inner-coupling among nodes and individual node dynamics. This is achieved by characterizing the left eigenvectors of the network state matrix. Special cases of this necessary and sufficient condition are also presented when the nodes are connected under path, cycle, star and wheel topologies. The results in this paper present the most general case in the scenario of networked systems, as all the system matrices, \textit{viz.} state, input, and output matrices, as well as the inner-coupling matrices, are  distinct and are of different dimensions.\\
	
	The paper is organized as follows: the  Introduction is followed by Section \ref{pre}, wherein the basics of controllability and observability of stand-alone systems are given. The model under consideration in this work is formulated in Section \ref{model}. Section \ref{res} forms the core of the paper, with all the main results stated here. The section gives two necessary and sufficient conditions for the controllability of the networked system, necessary and sufficient condition for the observability of the networked system and some necessary conditions for the controllability of the networked system. Variants of these results under certain specific network topologies are also presented. The results obtained are substantiated with numerical examples. Section \ref{con} concludes the paper, giving insights into future research directions.
	
	\section{Notations \& Preliminaries}\label{pre}
	
	\subsection{Notations}
	In this paper, $\mathbb{R}(\mathbb{C})$ denotes the set of all real (complex) numbers. $\mathbb{R}^{n}(\mathbb{C}^n)$ denotes the set of $n\times 1$ real (complex) vectors. $\mathbb{R}^{m\times n}(\mathbb{C}^{m\times n})$ denotes the set of all $m\times n$ real (complex) matrices. $diag\{a_1,\cdots, a_n\}$ denotes the $n\times n$ diagonal matrix with $a_1,\cdots, a_n$ as the diagonal entries and $blockdiag\{A_1,\cdots, A_N\}$ denotes the block-diagonal matrix with $A_1,\cdots, A_N$ as the diagonal blocks. $\otimes$ denotes the Kronecker product between matrices\cite{Horn-1}.
	
	\subsection{Control Theory}
	Consider the Linear Time-Invariant (LTI) system $\mathcal{S}$:
	\begin{align}
		\dot{x}(t)&=Ax(t)+Bu(t)\label{sa1}\\
		y(t)&=Cx(t)\label{sa2}
	\end{align} 
	where, $A\in \mathbb{R}^{n\times n}, B\in\mathbb{R}^{n\times p}, C\in\mathbb{R}^{m\times n}$ are the state, input and output matrices, respectively. Equation (\ref{sa1}) is the state equation and (\ref{sa2}) is the output equation.
	\begin{Definition}[\textbf{Controllability}]\cite{W.Terrell}
		The linear system (\ref{sa1}) is said to be \textbf{controllable} if, given any two states $x_0,x_f\in \mathbb{R}^n$ and a $t_f>0$, there exists an admissible control function $u(t)$, defined for $t\in [t_0,t_f]$, such that the corresponding solution of equation (\ref{sa1}) with initial condition $x(t_0)=x_0$ also satisfies $x(t_f)=x_f$.
	\end{Definition}
	\begin{Definition}[\textbf{Controllable Pair}]
		$(A,B)$ is said to be a controllable pair if the system (\ref{sa1}) is controllable.
	\end{Definition}
	The following are some of the fundamental necessary and sufficient conditions for controllability of the system (\ref{sa1}).
	\begin{thm}\cite{W.Terrell}\label{control}
		The LTI system $(A,B)$ is controllable if and only if one of the following conditions is satisfied:
		\begin{itemize}
			\item[i.] \textbf{Kalman's rank condition:}
			The controllability matrix 
			\begin{equation*}
				\mathcal{Q}(A,B)=rank[B|AB|\cdots|A^{n-1}B]
			\end{equation*} has full rank; i.e.,
			$rank[\mathcal{Q}(A,B)]=n$.
			\item[ii.] \textbf{PBH eigenvector test:}
			For every complex $\lambda$, the only $n\times 1$ vector $v$ that satisfies
			\begin{align*}
				v^*A=\lambda v^*,\ 
				v^*B=0
			\end{align*}
			is the zero vector $v=0$. Here $v^*$ is the conjugate transpose of $v$.

		\end{itemize}
		The PBH eigenvector test can be equivalently stated as follows:
		\begin{Nt}
			The matrix pair $(A,B)$ is controllable if and only if 
			$rank\begin{bmatrix}
				A-\lambda I &|&B
			\end{bmatrix}=n$ 
			for every eigenvalue $\lambda$ of $A$.
		\end{Nt}
	\end{thm}
	\begin{Definition}[\textbf{Observability}]\cite{W.Terrell}
		The linear system (\ref{sa1})-(\ref{sa2}) is said to be \textbf{observable} if, for any $x(t_0)=x_0$, there is a finite time $t_f>0$ such that knowledge of the input $u(t)$ and the output $y(t)$ for $t\in[t_0,t_f]$ uniquely determines $x_0$.
	\end{Definition}
	
	\begin{Definition}[\textbf{Observable Pair}]
		$(C,A)$ is said to be an observable pair if the system (\ref{sa1})-(\ref{sa2}) is observable.
	\end{Definition}
	
	The following are some of the fundamental necessary and sufficient conditions for the observability of the system (\ref{sa1})-(\ref{sa2}).
	\begin{thm}\cite{W.Terrell}\label{obsv}
		The LTI system $(C,A)$ is observable if and only if one of the following conditions is satisfied:
		
		\begin{itemize}
			\item[i.] \textbf{Kalman's rank condition:}
			The observability matrix 
			\begin{equation*}
				\mathcal{O}(C,A)=rank\begin{pmatrix}C\\--\\CA\\--\\\vdots\\--\\CA^{n-1}\end{pmatrix}
			\end{equation*}
			has full rank; i.e., $rank[\mathcal{O}(C,A)]=n$.
			
			\item[ii.] \textbf{PBH eigenvector test:}
			For every complex $\lambda$, the only $n\times 1$ vector $v$ that satisfies
			\begin{align*}
				Av=\lambda v,\ 
				Bv=0
			\end{align*}
			is the zero vector $v=0$.
			
		\end{itemize}
	\end{thm}
	
	\section{Model Formulation}\label{model}
	
	In this paper we consider the controllability and observability of a networked system with $N$ nodes, with each node having distinct dimension. Let the state, input and output vectors of the $i-$th node be denoted by $x_i, u_i$ and $y_i$,   respectively. Let $H_i$ denote the inner-coupling matrix for node $i$, which determines the inner interactions among the nodes. The individual node dynamics of the $i$-th node is given by:
	\begin{align}
		\dot{x_i}&=A_ix_i+\sum_{j=1}^{N}\beta_{ij}H_iy_j+\delta_iB_iu_i\label{sys1}\\
		y_i&=C_ix_i\label{sys2}
	\end{align}
	That is, 
	\begin{equation}
		\dot{x_i}=A_ix_i+\sum_{j=1}^{N}\beta_{ij}H_iC_jx_j+\delta_iB_iu_i
	\end{equation}
	Here, equation (\ref{sys1}) is called the state equation and equation (\ref{sys2}), the output equation. Also, $x_i\in \mathbb{R}^{n_i},y_i\in \mathbb{R}^m, u_i\in \mathbb{R}^{p_i}, A_i\in \mathbb{R}^{n_i\times n_i}, H_i\in \mathbb{R}^{n_i\times m}, C_i\in \mathbb{R}^{m\times n_i}, B_i\in \mathbb{R}^{n_i\times p_i},\forall\, i, 1\leq i\leq N$ and $\delta_i$ denotes the presence parameter of external control input to node $i$. That is,
	\begin{equation*}
		\delta_i=\begin{dcases}
			1& ;\textit{if node $i$ has external control input}\\
			0& ;otherwise
		\end{dcases}
	\end{equation*}
	The nodes are connected in a network topology $(L,\Delta)$, where the network topology matrix $L=[\beta_{ij}]_{i,j=1}^{N}$,
	\begin{equation*}
		\beta_{ij}\begin{dcases}
			\neq 0 &;\textit{if there is an edge from node $j$ to node $i$}\\
			= 0& ;otherwise
		\end{dcases}
	\end{equation*} 
	and the external input channel matrix $\Delta=diag\{\delta_1,\cdots,\delta_N\}$. The dynamics of the networked system is given by the following compact form:
	\begin{align}
		\dot{X}&=\mathcal{A} X+\mathcal{B} U \label{eq2}\\
		Y&=\mathcal{C}X \label{eq2-1}
	\end{align}
	where,
	\begin{align*}
		\mathcal{A}&=A+\Gamma \left(A=blockdiag\{A_1,...,A_N\}, \Gamma=(\beta_{ij}H_iC_j)_{i,j=1}^{N}\right)\\
		\mathcal{B}&=blockdiag\{\delta_1B_1,\delta_2B_2,...,\delta_NB_N\}\\
		\mathcal{C}&=blockdiag\{C_1,C_2,...,C_N\}
	\end{align*}
	
	\noindent Here, the state matrix $\mathcal{A}$ is an $\mathcal{N}\times \mathcal{N}$ matrix, where $\mathcal{N}=\sum_{i=1}^{N}n_i$ and the input matrix $\mathcal{B}$ is an $\mathcal{N}\times \mathcal{P}$ matrix, where $\mathcal{P}=\sum_{i=1}^{N}p_i$. When all the nodes have same dimension and the system matrices coincide, the networked system is homogeneous.
	
	\section{Main Results}\label{res}
	
	In this section, we give necessary and sufficient conditions for controllability and observability of the networked system and some easily verifiable necessary conditions for controllability of the networked system. These results give insights into the dependence of controllability and observability of the networked system on factors like network topology, interaction among nodes, individual node dynamics, and external control inputs.
	
	\subsection{Necessary and Sufficient Conditions for Controllability of the Networked System}
	
	The first theorem is a necessary and sufficient condition for controllability of the networked system, which is a reformulation of the PBH eigenvector test for controllability.
	\begin{thm}\cite{A.Thomas}\label{thm1}
		The networked system \eqref{eq2} is controllable if and only if for any $\lambda \in \mathbb{C}$, the system of matrix equations
		\begin{align}
			&v_i(\lambda I_{n_i}-A_i)-\sum_{j=1}^{N}\beta_{ji}v_jH_jC_i=0 \label{eq-cdtn-1}\\
			&\delta_iv_iB_i=0, i=1,...,N \label{eq-cdtn-2}
		\end{align} has only the zero solution.
	\end{thm}
	\begin{proof}
		By the PBH eigenvector test (Theorem \ref{control}(ii)), $(\mathcal{A},\mathcal{B})$ is controllable if and only if $v=0$ is the only solution to the system of matrix equations
		\begin{align*}
			v^*\mathcal{A}=\lambda v^*,\ 
			v^*\mathcal{B}=0
		\end{align*}
		Let $v=[v_1\,v_2\,...\,v_N]^T$ with $v_i\in \mathbb{C}^{1\times n_i}$, the above equations can be rewritten as:
		\begin{align*}
			[v_1\,v_2\,...\,v_N](\lambda I_{n_i}-A-\Gamma)&=0\\
			[v_1\,v_2\,...\,v_N]\mathcal{B}&=0
		\end{align*}
		which is equivalent to:
		\begin{align*}
			v_i(\lambda I_{n_i}-A_i)-\sum_{j=1}^{N}\beta_{ji}v_jH_jC_i&=0\\
			\delta_iv_iB_i&=0
		\end{align*}
	\end{proof}
	\begin{eg}\label{eg1}
		\begin{justify}
			Consider the networked system with the following nodes
			\begin{align*}
				A_1=\begin{pmatrix}
					1&2\\1&0
				\end{pmatrix}, A_2=\begin{pmatrix}
					1&0&0\\3&2&0\\4&1&-1
				\end{pmatrix},
				B_1=\begin{pmatrix}
					0\\1
				\end{pmatrix}, B_2=\begin{pmatrix}
					1\\2\\-1
				\end{pmatrix}\\
				C_1=\begin{pmatrix}
					1&-1
				\end{pmatrix}, C_2=\begin{pmatrix}
					1&-1&1
				\end{pmatrix},
				H_1=\begin{pmatrix}
					0\\-1	\end{pmatrix}, H_2=\begin{pmatrix}
					1\\1\\1
				\end{pmatrix}
			\end{align*} 
			
			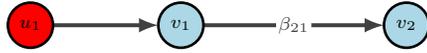
\begin{figure}[h!]\label{fig1}
				\center
				
				\begin{tikzpicture}
					\Vertex[x=1,y=0,label=$v_1$]{a1}
					\Vertex[x=4,y=0,label=$v_2$]{a2}
					\Vertex[x=-1,y=0,color=red,label=$u_1$]{b1}
					\Edge[Direct](b1)(a1)
					\Edge[Direct,label=$\beta_{21}$](a1)(a2)
				\end{tikzpicture} 
				\caption{Controllable networked system with dynamics as in Example \ref{eg1}}
			\end{figure}
			Choose $\beta_{ij}$ and $\delta_i$ such that
			$L=\begin{pmatrix}
				0&0\\1&0
			\end{pmatrix}$ and 
			$\Delta=\begin{pmatrix}
				1&0\\0&0
			\end{pmatrix}$.	Then, the networked system can be written in the compact form \eqref{eq2} with
			$$\mathcal{A}=\begin{pmatrix}
				1&2&0&0&0\\1&0&0&0&0\\1&-1&1&0&0\\1&-1&3&2&0\\1&-1&4&1&-1
			\end{pmatrix}, \mathcal{B}=\begin{pmatrix}
				0&0\\1&0\\0&0\\0&0\\0&0
			\end{pmatrix}$$ 
			By Kalman's rank condition (Theorem \ref{control} (i)), the networked system is controllable. Now, let us verify Theorem \ref{thm1}. Let $v_1=\begin{pmatrix}
				v_{11}&v_{12}
			\end{pmatrix}\in \mathbb{C}^{1\times 2}$ and $v_2=\begin{pmatrix}
				v_{21}&v_{22}&v_{23}
			\end{pmatrix}\in \mathbb{C}^{1\times 3}$ satisfy \eqref{eq-cdtn-1}-\eqref{eq-cdtn-2}:
			\begin{align*}
				v_1(\lambda I_2-A_1)-v_2H_2C_1&=0\\
				v_1B_1&=0\\
				v_2(\lambda I_3-A_2)&=0
			\end{align*} 
			for $\lambda \in \mathbb{C}$. This gives rise to the following system of equations:
			\begin{align*}
				(\lambda-1)v_{11}+v_{12}-v_{21}-v_{22}-v_{23}&=0\\
				2v_{11}+\lambda v_{12}+v_{21}+v_{22}+v_{23}&=0\\
				v_{12}&=0\\
				(\lambda-1)v_{21}+3v_{22}+4v_{23}&=0\\
				(\lambda-2)v_{22}+v_{23}&=0\\
				(\lambda+1)v_{23}&=0
			\end{align*} 
			Clearly, $v=(v_1,v_2)=0$ is the only solution.
		\end{justify}
	\end{eg}
	
	For a heterogeneous networked system with same node dimensions and inner-coupling matrices, i.e.,
	\begin{align}
		\dot{x_i}&=A_ix_i+\sum_{j=1}^{N}\beta_{ij}HC_jx_j+\delta_iB_iu_i\label{het2-1}
	\end{align}
	Xiang \textit{et al.}\cite{L.Xiang} established a necessary and sufficient condition for controllability, which becomes a special case of Theorem \ref{thm1} as stated in the following corollary.
	
	\begin{cor}[Xiang et al., Theorem 1, \cite{L.Xiang}]
		The networked system with individual node dynamics as in \eqref{het2-1} is controllable if and only if
		\begin{align}
			v_i^*(\lambda I_{n}-A_i)-\sum_{j=1,j\neq i}^{N}\beta_{ji}v_jHC_i=0\\
			\delta_iv_i^*B_i=0, i=1,...,N
		\end{align} 
		has a unique solution $v_i=0$ for any complex number $\lambda$ and for all $i=1,\cdots,N$.
	\end{cor}
	Wang \textit{et al.} \cite{L.Wang} studied the homogeneous system, i.e., $n_i=n\ \text{and}\ A_i=A, B_i=B, C_i=C, H_i=H \,\ \forall\, i$. That is,
	\begin{align}
		\dot{x_i}&=Ax_i+\sum_{j=1}^{N}\beta_{ij}Hy_j+\delta_iBu_i\label{wang}\\
		y_i&=Cx_i\label{wang2}
	\end{align}
	The necessary and sufficient condition obtained in\cite{L.Wang} becomes a special case of Theorem \ref{thm1} as stated in the following corollary.
	\begin{cor} [Wang \textit{et al.}, Theorem 2, \cite{L.Wang}]
		The homogeneous networked system \eqref{wang}-\eqref{wang2} is controllable if and only
		if, for any $\lambda \in \mathbb{C}$,  $F =0$ is the only matrix solution of the simultaneous equations:
		\begin{align}
			\Delta^TFB =0\ \text{and}\ 
			L^TFHC=F(\lambda I-A)
		\end{align}
	\end{cor}
	\begin{proof}
		Let $F=\begin{bmatrix}
			v_1&v_2&\ldots&v_N
		\end{bmatrix}$ where $v_i$ is as in Theorem \ref{thm1}. 
		\begin{align*}
			\eqref{eq-cdtn-1}&\implies F(\lambda I_n-A)-L^TFHC=0\\
			\eqref{eq-cdtn-2}&\implies \Delta^T FB=0
		\end{align*}
		Hence the corollary follows from Theorem \ref{thm1}.
	\end{proof}
	Theorem \ref{thm1} gives little information about the dependence of network controllability on factors like network topology, inner interactions among nodes, and individual node dynamics. Since it is computationally demanding as it involves solving higher-order matrix equations, we have formulated another necessary and sufficient condition for network controllability which requires  properties of the left eigenvectors of $\mathcal{A}$.
	
	\begin{thm}\label{thm2}
		Consider the networked system \eqref{eq2}-\eqref{eq2-1}. Let $\sigma(A_{i}+\beta_{ii}H_{i}C_{i})=\{ \mu_{i}^{1}, \mu_{i}^{2},\cdots,\mu_{i}^{q_i}\}, 1\leq i\leq N$ and $\zeta_{ij}^k$ denote the left eigenvectors of $A_{i}+\beta_{ii}H_{i}C_{i}$ corresponding to the eigenvalue $\mu_i^j$, where, $1\leq k\leq \gamma_{ij}$, $\gamma_{ij}$ denotes the geometric multiplicity of $\mu_i^j$ as an eigenvalue of $A_{i}+\beta_{ii}H_iC_i$. Suppose whenever $\beta_{ij}\neq 0, H_{i}C_{j}=0, \forall\, i>j$ and $\zeta_{ij}^k$ is orthogonal to $H_{i}C_{j}, \forall \,i<j$, $k=1,\cdots,\gamma_{ij}$. Then $\sigma(\mathcal{A})=\bigcup_{i=1}^{N}\sigma(A_i+\beta_{ii}H_iC_i)$ and the left eigenvectors of $\mathcal{A}$, corresponding to the eigenvalue $\mu_{i}^{j}$ are precisely $\{\eta_{ij}^{1},\cdots,\eta_{ij}^{\gamma_{ij}}\}$, where $\eta_{ij}^k=[\underbrace{0}_{1\times n_1}\cdots\underbrace{0}_{ 1\times n_{i-1}} \ \underbrace{\zeta_{ij}^{k}}_{1\times n_i}\ \underbrace{0}_{1\times n_{i+1}}\cdots \underbrace{0}_{1\times n_N}], 1\leq k\leq \gamma_{ij}$.
	\end{thm}
	\begin{proof}
		For the networked system $(\mathcal{A}, \mathcal{B})$, $\mathcal{A}$ is given by
		$$\begin{pmatrix}
			A_1+\beta_{11}H_1C_1&\beta_{12}H_1C_2&\cdots&\beta_{1N}H_1C_N\\
			\beta_{21}H_2C_1&A_2+\beta_{22}H_2C_2&\cdots&\beta_{2N}H_2C_N\\
			\vdots&\vdots&\ddots&\vdots\\
			\beta_{N1}H_NC_1&\beta_{N2}H_NC_2&\cdots&A_N+\beta_{NN}H_NC_N\\
		\end{pmatrix}$$
		If $H_iC_j=0$ whenever $\beta_{ij}\neq 0$ $\forall \,i>j$, $\mathcal{A}$ becomes 
		$$\begin{pmatrix}
			A_1+\beta_{11}H_1C_1&\beta_{12}H_1C_2&\cdots&\beta_{1N}H_1C_N\\
			0&A_2+\beta_{22}H_2C_2&\cdots&\beta_{2N}H_2C_N\\
			\vdots&\quad&\ddots&\vdots\\
			0&0&\cdots&A_N+\beta_{NN}H_NC_N\\
		\end{pmatrix}$$ 
		Here $\mathcal{A}$ is a block upper triangular matrix with $A_i+\beta_{ii}H_iC_i, 1\leq i\leq N$ as the diagonal blocks. Hence, $$\sigma(\mathcal{A})=\bigcup_{i=1}^{N}\sigma(A_i+\beta_{ii}H_iC_i) =\{\mu_{1}^1,\cdots,\mu_{1}^{q_1},\cdots,\mu_{N}^1,\cdots,\mu_{N}^{q_N}\} $$
		Let $\zeta_{ij}^{k}, k=1,\cdots,\gamma_{ij}$ denote the left eigenvectors of $A_i+\beta_{ii}H_iC_i$, corresponding to the eigenvalue $\mu_{i}^j$. Take  $\eta_{ij}^k=[\underbrace{0}_{1\times n_1}\cdots\underbrace{0}_{ 1\times n_{i-1}} \ \underbrace{\zeta_{ij}^{k}}_{1\times n_i}\ \underbrace{0}_{1\times n_{i+1}}\cdots \underbrace{0}_{1\times n_N}], 1\leq k\leq \gamma_{ij}$. Then,
		\begin{align*}\eta_{ij}^k\mathcal{A}&=\begin{bmatrix}
				0&\cdots&0&\zeta_{ij}^k&0&\cdots&0
			\end{bmatrix}\mathcal{A}\\
			&=\begin{bmatrix}
				0&\cdots&0&\zeta_{ij}^k(A_i+\beta_{ii}H_iC_i)&0&\cdots&0
			\end{bmatrix}\\
			&=\begin{bmatrix}
				0&\cdots&0&\mu_{i}^{j}\zeta_{ij}^k&0&\cdots&0
			\end{bmatrix}=\mu_{i}^{j}\eta_{ij}^k\end{align*} 
		hence $\eta_{ij}^k$ are left eigenvectors of $\mathcal{A}$ for $1\leq i\leq N, 1\leq j\leq q_i, 1\leq k\leq \gamma_{ij}$.\\
		
		\noindent Now, we will prove that the only left eigenvectors of $\mathcal{A}$ are in the form $\eta_{ij}^k=[\underbrace{0}_{1\times n_1}\cdots\underbrace{0}_{ 1\times n_{i-1}} \ \underbrace{\zeta_{ij}^{k}}_{1\times n_i}\ \underbrace{0}_{1\times n_{i+1}}\cdots \underbrace{0}_{1\times n_N}], 1\leq k\leq \gamma_{ij}$. Let $v\in \mathbb{R}^{1\times \mathcal{N}}$ be a left eigenvector of $\mathcal{A}$ corresponding to some eigenvalue $\mu_{i}^j$. Let $v=\begin{bmatrix}
			v_1&\cdots&v_N
		\end{bmatrix}, v_i\in \mathbb{R}^{1\times n_i}$.
		From $v\mathcal{A}=\mu_{i}^jv$, we get the following equations
		\begin{align*}
			v_1(A_1+\beta_{11}H_1C_1)&=\mu_{i}^j v_1\tag{i}\\
			\beta_{12}v_1H_1C_2+v_2(A_2+\beta_{22}H_2C_2)&=\mu_{i}^j v_2\tag{ii}\\
			&\vdots\nonumber\\
			\beta_{1N}v_1H_1C_N+\cdots+v_N(A_N+\beta_{NN}H_NC_N)&=\mu_{i}^jv_N \tag{N}
		\end{align*}
		The $i-$th equation, $1\leq i\leq N$ implies that either $v_i=0$ or $v_i$ is a left eigenvector of $A_i+\beta_{ii}H_iC_i$ corresponding to the eigenvalue $\mu_{i}^j$ as $\zeta_{ij}^k$ is orthogonal to $H_{i}C_{j}$; i.e., $v_i=\zeta_{ij}^k$ for some $k,1\leq k\leq \gamma_{ij}$. Hence, $v=\eta_{ij}^k$.
		
	\end{proof}
	Now, let us see some examples that illustrate Theorem \ref{thm2}.
	\begin{eg}\label{eg2}
		Consider a networked system with three nodes with the following dynamics:
		\begin{align*}
			&A_1=\begin{pmatrix}
				1&2\\0&1
			\end{pmatrix}, A_2=\begin{pmatrix}
				1&0&2\\1&3&2\\-1&-1&0
			\end{pmatrix}, 	A_3=\begin{pmatrix}
				1&0\\0&0
			\end{pmatrix}\\
			&B_1=\begin{pmatrix}
				0\\1
			\end{pmatrix}, B_2=\begin{pmatrix}
				0\\1\\1
			\end{pmatrix}, B_3=\begin{pmatrix}
				1\\1
			\end{pmatrix} \\
			&C_1=\begin{pmatrix}
				1&0\\1&0
			\end{pmatrix},
			C_2=\begin{pmatrix}
				1&0&1\\1&0&1
			\end{pmatrix},  C_3=\begin{pmatrix}
				0&0\\1&0
			\end{pmatrix}\\
			&H_1=\begin{pmatrix}
				1&0\\0&0
			\end{pmatrix},
			H_2=\begin{pmatrix}
				1&-1\\0&0\\0&0
			\end{pmatrix}
			H_3=\begin{pmatrix}
				1&-1\\0&0
			\end{pmatrix} 	
		\end{align*}
		Let   $L=\begin{pmatrix}
			0&1&0\\1&0&0\\0&1&0
		\end{pmatrix}\  \text{and}\ \Delta=\begin{pmatrix}
			1&0&0\\0&1&0\\0&0&1
		\end{pmatrix}$. 
		Then, the networked system can be written in the compact form \eqref{eq2} with
		$$\mathcal{A}=\begin{pmatrix}
			1&2&1&0&1&0&0\\0&1&0&0&0&0&0\\0&0&1&0&2&0&0\\0&0&1&3&2&0&0\\0&0&-1&-1&0&0&0\\0&0&0&0&0&1&0\\0&0&0&0&0&0&0
		\end{pmatrix}, \mathcal{B}=\begin{pmatrix}
			0&0&0\\1&0&0\\0&0&0\\0&1&0\\0&1&0\\0&0&1\\0&0&1
		\end{pmatrix}$$
		We can compute the eigenspectrum of $\mathcal{A}$ as $$\sigma(A)=\{1,1,1+\sqrt{2}i,1-\sqrt{2}i,2,0,1\}$$
		Table \ref{table1} lists the eigenvalues and left eigenvectors of the state matrices $A_1, A_2$ and $A_3$.\\
		\begin{table}[h!]
			\caption{Eigenvalues \& Left eigenvectors of state matrices $A_1,A_2,A_3$ in Example \ref{eg2}}
			\label{table1}
			\setlength{\tabcolsep}{2pt}
			\begin{tabular}{|p{25pt}|p{75pt}|p{115pt}|}
				\hline
				Matrix&Eigenvalues&Left eigenvectors\\
				\hline
				$A_1$&$\mu_1^1=1$&$\zeta_{11}^1=\begin{pmatrix}
					0\\1
				\end{pmatrix}^T$\\
				\hline
				$A_2$&$\mu_2^1=1+i\sqrt{2},$\par$\mu_2^2=1-i\sqrt{2},$\par$\mu_2^3=2$&$\zeta_{21}^1=\begin{pmatrix}-0.14- 0.39i\\-0.28- 0.19i\\-0.84\end{pmatrix}^T,$\par$\zeta_{22}^1=\begin{pmatrix}-0.14 + 0.39i\\-0.28+ 0.19i\\-0.84\end{pmatrix}^T,$\par$\zeta_{23}^1=\begin{pmatrix} 0\\-0.7071\\-0.7071\end{pmatrix}^T$\\
				\hline
				$A_3$&$\mu_3^1=0,\mu_3^2=1$&$\zeta_{31}^1=\begin{pmatrix}
					0\\1
				\end{pmatrix}^T,\zeta_{32}^1=\begin{pmatrix}
					1\\0
				\end{pmatrix}^T$\\
				\hline
			\end{tabular}
		\end{table}\\
		From Table \ref{table1}, we can see that the eigenspectrum of $\mathcal{A}$ is the union of the eigenspectra of $A_1, A_2$ and $A_3$. That is, $$\sigma(\mathcal{A})=\sigma(A_1)\cup\sigma(A_2)\cup\sigma(A_3)$$
		Now, let us compute the left eigenvectors of $\mathcal{A}$.\\
		$\bullet$ $\eta_{11}^1=\begin{pmatrix}
			0 &
			1 &
			0 &
			0 &
			0 &
			0 &
			0
		\end{pmatrix}^T$ is the left eigenvector of $\mathcal{A}$ corresponding to the eigenvalue $\mu_1^1=1$. \\
		$\bullet$ $\eta_{21}^1=(
		0\  
		0 \ 
		-(0.14+0.39i) \ 
		-(0.28+ 0.19i) \ 
		-0.84 \   
		0  \	0
		)^T$ is the left eigenvector of $\mathcal{A}$ corresponding to the eigenvalue $\mu_2^1=1+\sqrt{2}i$. \\
		$\bullet$ $\eta_{22}^1=(
		0 \ 
		0 \ 
		-0.14 + 0.39i\ 
		-0.28 + 0.19i\ 
		-0.84\ 
		0\ 
		0
		)^T$ is the left eigenvector of $\mathcal{A}$ corresponding to the eigenvalue $\mu_2^2=1-\sqrt{2}i$ \\
		$\bullet$  $\eta_{23}^1=\begin{pmatrix}
			0 &
			0 &
			0 &
			-0.7071 &
			-0.7071 &
			0&
			0
			
		\end{pmatrix}^T$ is the left eigenvector of $\mathcal{A}$ corresponding to the eigenvalue $\mu_2^3=2$.\\
		$\bullet$ $\eta_{31}^1=\begin{pmatrix}
			0  & 0 & 0 & 0 & 0 & 0 & 1
		\end{pmatrix}^T$ is the left eigenvector of $\mathcal{A}$ corresponding to the eigenvalue $\mu_3^1=0$ and
		$\eta_{32}^1=\begin{pmatrix}
			0 & 0 & 0 & 0 & 0 & 1 & 0
		\end{pmatrix}^T$ is the left eigenvector of $\mathcal{A}$ corresponding to the eigenvalue $\mu_3^2=1$

		We can see that the left eigenvectors of $\mathcal{A}$ are of the form $\eta_{ij}^k=[\underbrace{0}_{1\times n_1}\cdots\underbrace{0}_{ 1\times n_{i-1}} \ \underbrace{\zeta_{ij}^{k}}_{1\times n_i}\ \underbrace{0}_{1\times n_{i+1}}\cdots \underbrace{0}_{1\times n_N}], 1\leq k\leq \gamma_{ij}, 1\leq i \leq n, 1 \leq j\leq q_i$ and $\zeta_{ij}^{k}$ is the left eigenvector of $A_i$ corresponding to the eigenvalue $\mu_i^j$. Thus, Theorem \ref{thm2} is verified.
	\end{eg}
	
	The left eigenvectors of $\mathcal{A}$ being characterized in terms of the left eigenvectors of state matrices of the individual nodes, we obtain a verifiable necessary and sufficient condition for controllability of the networked system \eqref{eq2}-\eqref{eq2-1}.
	\begin{thm}\label{thm3}
		Consider the networked system \eqref{eq2}-\eqref{eq2-1}. Let $\sigma(A_{i}+\beta_{ii}H_{i}C_{i})=\{ \mu_{i}^{1}, \mu_{i}^{2},\cdots,\mu_{i}^{q_i}\}, 1\leq i\leq N$ and $\zeta_{ij}^k$ denote the left eigenvectors of $A_{i}+\beta_{ii}H_{i}C_{i}$ corresponding to the eigenvalue $\mu_i^j$ where $1\leq k\leq \gamma_{ij}$, $\gamma_{ij}$ denotes the geometric multiplicity of $\mu_i^j$ as an eigenvalue of $A_{i}+\beta_{ii}H_iC_i$. Suppose that whenever $\beta_{ij}\neq 0, H_{i}C_{j}=0, \forall i>j$ and $\zeta_{ij}^k$ is orthogonal to $H_{i}C_{j}, \forall i<j$, $k=1,\cdots,\gamma_{ij}$. Then the networked system is controllable if and only if:
		\begin{itemize}
			\item[(i)] $(A_i+\beta_{ii}H_iC_i,B_i)$ is controllable $\forall\, i, 1\leq i\leq N$.
			\item [(ii)] Each node is under external control, i.e., $\delta_i\neq 0, \forall i=1,\ldots,N$.
		\end{itemize} 
	\end{thm}
	\begin{proof}
		First, assume $(\mathcal{A},\mathcal{B})$ is controllable, then by PBH eigenvector test (Theorem \ref{control}(ii)), $\eta_{ij}^k \mathcal{B}\neq 0$ for $1\leq i\leq N, 1\leq j\leq q_i, 1\leq k\leq \gamma_{ij}$
		which implies that $\delta_i\zeta_{ij}^kB_i\neq 0, \forall i=1,\cdots,N$. That is,
		\begin{itemize}
			\item $\delta_i\neq 0$, i.e., node $i$ has an external control input and
			\item $\zeta_{ij}^kB_i\neq 0, \forall i=1,\cdots,N$, i.e., there is no left eigenvector of $A_i+\beta_{ii}H_iC_i$ that is orthogonal to $B_i$
		\end{itemize}  
		\noindent Hence $(A_i+\beta_{ii}H_iC_i,B_i)$ is a controllable pair $\forall\, i=1,\cdots,N$ and each node is under external control.\\
		
		\noindent Conversely, assume that $(\mathcal{A},\mathcal{B})$ is an uncontrollable pair. Further by PBH eigenvector test, there exists $\eta_{ij}^k\neq 0$ such that $\eta_{ij}^k \mathcal{B}=0$ and 
		\begin{align*}
			\eta_{ij}^k \mathcal{B}=0
			\implies \delta_i(\zeta_{ij}^k)B_i=0
			\implies \delta_i=0 \ \text{or}\  (\zeta_{ij}^k)B_i=0
		\end{align*}
		That is, either $(A_i+\beta_{ii}H_iC_i,B_i)$ is an uncontrollable pair or node $i$ is not under external control.
	\end{proof}
	This theorem is more efficient as compared to Theorem \ref{thm1} as it does not involve solving matrix equations. Also, the dependence of network controllability on various factors as stated above is clearly brought into the conditions.
	\begin{eg}\label{eg4}
		Consider the networked system in Example \ref{eg2}. Here, $\beta_{ii}=0, i=1,2,3$ and $(A_1,B_1), (A_2,B_2)$ and $(A_3,B_3)$ are controllable pairs. Also note that $H_2C_1=H_3C_2=0$ and $\beta_{31}=0$. Further, here $\zeta_{11}^1=\begin{pmatrix}
			0&1
		\end{pmatrix}$, which is orthogonal to $H_1C_2$ and $\beta_{13}=\beta_{23}=\beta_{33}=0$. Hence all the conditions of Theorem \ref{thm3} are satisfied and the networked system $(\mathcal{A}, \mathcal{B})$ is controllable .
		\begin{figure}[h!]\label{fig2}
			\center
			
			\begin{tikzpicture}
				\Vertex[x=1,y=0,label=$v_1$]{a1}
				\Vertex[x=3,y=0,label=$v_2$]{a2}
				\Vertex[x=5,y=0,label=$v_3$]{a3}
				\Vertex[x=-0.5,y=1,color=red,label=$u_1$]{b1}
				\Vertex[x=4,y=1,color=red,label=$u_2$]{b2}
				\Vertex[x=6,y=1,color=red,label=$u_3$]{b3}
				\Edge[Direct](b1)(a1)
				\Edge[Direct](b2)(a2)
				\Edge[Direct](b3)(a3)
				\Edge[bend=20,Direct,label=$\beta_{21}$](a1)(a2)
				\Edge[bend=20,Direct,label=$\beta_{12}$](a2)(a1)
				\Edge[Direct,label=$\beta_{32}$](a2)(a3)
			\end{tikzpicture} 
			\caption{Controllable networked system with dynamics as in Example \ref{eg4}}
		\end{figure}
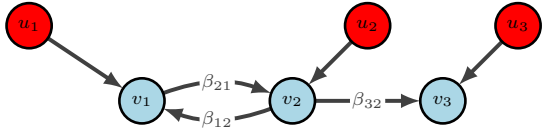
	\end{eg}
	\begin{eg}\label{eg5}
		Consider the same networked system as in Example \ref{eg2}, except that no external control input is applied to node 3, i.e., $\Delta=\begin{pmatrix}
			1&0&0\\0&1&0\\0&0&0
		\end{pmatrix}$. 
		\begin{figure}[h!]\label{fig5}
			\center
			
			\begin{tikzpicture}
				\Vertex[x=1,y=0,label=$v_1$]{a1}
				\Vertex[x=3,y=0,label=$v_2$]{a2}
				\Vertex[x=5,y=0,label=$v_3$]{a3}
				\Vertex[x=-0.5,y=1,color=red,label=$u_1$]{b1}
				\Vertex[x=4,y=1,color=red,label=$u_2$]{b2}
				\Edge[Direct](b1)(a1)
				\Edge[Direct](b2)(a2)
				\Edge[bend=20,Direct,label=$\beta_{21}$](a1)(a2)
				\Edge[bend=20,Direct,label=$\beta_{12}$](a2)(a1)
				\Edge[Direct,label=$\beta_{32}$](a2)(a3)
			\end{tikzpicture} 
			\caption{Uncontrollable networked system with dynamics as in Example \ref{eg5}}
		\end{figure}
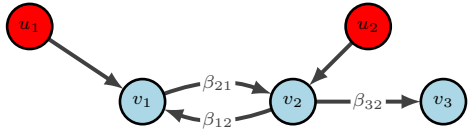\\
		Then, the networked system can be written in the compact form \eqref{eq2} with
		$$\mathcal{A}=\begin{pmatrix}
			1&2&1&0&1&0&0\\0&1&0&0&0&0&0\\0&0&1&0&2&0&0\\0&0&1&3&2&0&0\\0&0&-1&-1&0&0&0\\0&0&0&0&0&1&0\\0&0&0&0&0&0&0
		\end{pmatrix},\mathcal{B}=\begin{pmatrix}
			0&0&0\\1&0&0\\0&0&0\\0&1&0\\0&1&0\\0&0&0\\0&0&0
		\end{pmatrix}$$  Here, $(\mathcal{A},\mathcal{B})$ is an uncontrollable pair. This example illustrates the importance of condition (ii) in Theorem \ref{thm3}.
	\end{eg}
	\begin{eg}\label{eg6}
		Consider the 3-node networked system where
		\begin{align*}
			&A_1=\begin{pmatrix}
				1&2\\0&1
			\end{pmatrix}, A_2=\begin{pmatrix}
				1&0&0\\0&1&0\\0&0&2
			\end{pmatrix}, A_3=\begin{pmatrix}
				1&0\\0&0
			\end{pmatrix}\\
			&B_1=\begin{pmatrix}
				0\\1
			\end{pmatrix}, B_2=\begin{pmatrix}
				0\\0\\1
			\end{pmatrix}, B_3=\begin{pmatrix}
				1\\1
			\end{pmatrix}\\
			&C_1=\begin{pmatrix}
				1&0\\1&0
			\end{pmatrix}, C_2=\begin{pmatrix}
				1&0&1\\1&0&1
			\end{pmatrix}, C_3=\begin{pmatrix}
				0&0\\1&0
			\end{pmatrix}\\
			&H_1=\begin{pmatrix}
				1&0\\0&0
			\end{pmatrix}, H_2=\begin{pmatrix}
				1&-1\\0&0\\0&0
			\end{pmatrix}, H_3=\begin{pmatrix}
				1&-1\\0&0
			\end{pmatrix} 	
		\end{align*} 
		Choose $\beta_{ij}$ and $\delta_i$ such that $$L=\begin{pmatrix}
			0&1&0\\1&0&0\\0&1&0
		\end{pmatrix}\ \text{and}\ \Delta=\begin{pmatrix}
			1&0&0\\0&1&0\\0&0&1
		\end{pmatrix}$$
		Then, the networked system can be written in the compact form \eqref{eq2} with
		$$\mathcal{A}=\begin{pmatrix}
			1&2&1&0&1&0&0\\0&1&0&0&0&0&0\\0&0&1&0&0&0&0\\0&0&0&1&0&0&0\\0&0&0&0&2&0&0\\0&0&0&0&0&1&0\\0&0&0&0&0&0&0
		\end{pmatrix},\mathcal{B}=\begin{pmatrix}
			0&0&0\\1&0&0\\0&0&0\\0&0&0\\0&1&0\\0&0&1\\0&0&1
		\end{pmatrix}$$ Here, $\beta_{ii}=0, i=1,2,3$ and $(A_1,B_1)$ and $(A_3,B_3)$ are controllable pairs, wheras $(A_2,B_2)$ is an uncontrollable pair. Also, note that $H_2C_1=H_3C_2=0$ and $\beta_{31}=0$. Further, here $\zeta_{11}^1=\begin{pmatrix}
			0&1
		\end{pmatrix}$, which is orthogonal to $H_1C_2$ and $\beta_{13}=\beta_{23}=\beta_{33}=0$.  Networked system $(\mathcal{A},\mathcal{B})$ is not controllable. This example illustrates the importance of condition (i) in Theorem \ref{thm3}.
		\begin{figure}[h!]\label{fig3}
			\center
			
			\begin{tikzpicture}
				\Vertex[x=1,y=0,label=$v_1$]{a1}
				\Vertex[x=3,y=0,label=$v_2$]{a2}
				\Vertex[x=5,y=0,label=$v_3$]{a3}
				\Vertex[x=-0.5,y=1,color=red,label=$u_1$]{b1}
				\Vertex[x=4,y=1,color=red,label=$u_2$]{b2}
				\Vertex[x=6,y=1,color=red,label=$u_3$]{b3}
				\Edge[Direct](b1)(a1)
				\Edge[Direct](b2)(a2)
				\Edge[Direct](b3)(a3)
				\Edge[bend=20,Direct,label=$\beta_{21}$](a1)(a2)
				\Edge[bend=20,Direct,label=$\beta_{12}$](a2)(a1)
				\Edge[Direct,label=$\beta_{32}$](a2)(a3)
			\end{tikzpicture} 
			\caption{Uncontrollable networked system with dynamics as in Example \ref{eg6}}
		\end{figure}
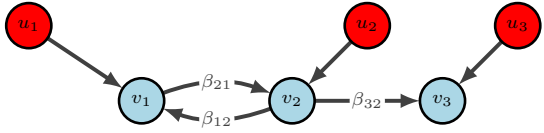
	\end{eg}
	
	\noindent Consider the networked system \eqref{eq2}-\eqref{eq2-1} where $n_i=n$ and $C_i=I_n,\forall\, i$.
	The dynamics of the $i-$th node is given by 
	\begin{equation}\label{ns3}
		\dot{x_i}=A_ix_i+\sum_{j=1}^{N}\beta_{ij}H_ix_j+\delta_iB_iu_i; i=1,\cdots,N
	\end{equation}
	and the networked system, under topology $L$ is given by the following compact form
	\begin{align}
		\dot{X}&=\mathcal{A} X+\mathcal{B} U
	\end{align}
	where,
	\begin{align*}
		\mathcal{A}&=A+\Gamma \quad\left(A=diag\{A_1,A_2,...,A_N\}, \Gamma=(\beta_{ij}H_i)_{i,j=1}^{N}\right)\\
		\mathcal{B}&=diag\{\delta_1B_1,\delta_2B_2,...,\delta_NB_N\}
	\end{align*}
	Theorems (\ref{thm2}) and (\ref{thm3}) apply to this system under triangular network topology as considered in Ajayakumar \& George\cite{A.Ajayakumar-3}.
	\begin{cor}[Ajayakumar \& George , Theorem 4, \cite{A.Ajayakumar-3}]\label{cor1}
		Assume that $L$ is an upper/ lower triangular matrix. Let $\sigma(A_i+\beta_{ii}H_i)=\{\mu_{i}^1,\mu_{i}^2,\cdots,\mu_{i}^{q_i}\}$ be the set of eigenvalues of $A_i+\beta_{ii}H_i$. Then, the set of all eigenvalues of $\mathcal{A}$ is given by $$\sigma(\mathcal{A})=\{\mu_{1}^1,\mu_{1}^2,\cdots,\mu_{1}^{q_1},\cdots,\mu_{N}^1,\mu_{N}^2,\cdots,\mu_{N}^{q_N}\}$$ Let $\zeta_{ij}^k,k=1,2,\cdots,\gamma_{ij}$ be the left eigenvectors of $A_i+\beta_{ii}H_i$ associated with the eigenvalue $\mu_{i}^j$, where $\gamma_{ij}$ is the geometric multiplicity of $\mu_{i}^j$. If $\zeta_{ij}^kH_i=0$, for $i=1,2,\cdots,N-1, j=1,2,\cdots,q_i, k=1,2,\cdots, \gamma_{ij}$, then $e_i\otimes \zeta_{ij}^1,e_i\otimes \zeta_{ij}^2,\cdots, e_i\otimes \zeta_{ij}^{\gamma_{ij}} $ are the left eigenvectors of $\mathcal{A}$ associated with the eigenvalues $\mu_{i}^j$.
	\end{cor}
	\begin{proof}
		From the proof of Theorem \ref{thm2}, $$\eta_{ij}^k=[\underbrace{0}_{1\times n_1}\cdots\underbrace{0}_{ 1\times n_{i-1}} \ \underbrace{\zeta_{ij}^{k}}_{1\times n_i}\ \underbrace{0}_{1\times n_{i+1}}\cdots \underbrace{0}_{1\times n_N}]$$ where $1\leq k\leq \gamma_{ij}, 1\leq i \leq n, 1 \leq j\leq q_i$ and $\zeta_{ij}^{k}$ is the left eigenvector of $A_i$ corresponding to the eigenvalue $\mu_i^j$. Since $n_i=n, \forall i$, 
		$$\eta_{ij}^k=[\underbrace{0}_{1\times n}\cdots\underbrace{0}_{ 1\times n} \ \underbrace{\zeta_{ij}^{k}}_{1\times n}\ \underbrace{0}_{1\times n}\cdots \underbrace{0}_{1\times n}]=e_i\otimes \zeta_{ij}^k$$ Hence the corollary follows from Theorem \ref{thm2}.
	\end{proof}
	The next corollary also follows from the fact that $n_i=n, \forall i=1,2,\ldots,n$.
	\begin{cor}[Ajayakumar \& George , Theorem 5, \cite{A.Ajayakumar-3}]\label{cor2}
		Let $L$ be an upper/lower triangular matrix. Suppose the left eigenvectors of $A_i+\beta_{ii}H_i$ satisfy the conditions given in Corollary (\ref{cor1}), then the networked system \ref{ns3} is controllable if and only if:
		\begin{itemize}
			\item [(i)]$(A_i+\beta_{ii}H_i, B_i)$ is a controllable pair for all $i=1,2,\cdots, N$
			\item [(ii)] Every node have external control input, i.e., $\delta_i\neq 0,\forall i$
		\end{itemize}
	\end{cor}
	
	\subsection{A Necessary and Sufficient Condition for Observability of the Networked System}
	
	We now present a necessary and sufficient condition for the observability of the networked system (\ref{sys1})-(\ref{sys2}).
	\begin{thm}\label{thm4}
		The networked system (\ref{sys1})-(\ref{sys2}) is observable if and only if all $(C_i,A_i)$s are observable pairs.
	\end{thm}
	\begin{proof}
		Suppose that $(\mathcal{C},\mathcal{A})$ is an observable pair. 
		Let $(C_i,A_i)$ be an unobservable pair for some $i$.
		Then, by the PBH eigenvector test, there exists $v(\neq 0)\in \mathbb{R}^{n_i}$ such that 
		\begin{align*}
			A_iv=\lambda v\ \text{and} \ 
			C_iv=0
		\end{align*} 
		where, $\lambda$ is a scalar.
		Now, consider a 
		$\mathcal{N} \times 1$ vector $$V=[\underbrace{0}_{1\times n_1}\cdots\underbrace{0}_{1\times n_{i-1}}\ \underbrace{v}_{1\times n_i}\ \underbrace{0}_{1\times n_{i+1}}\cdots\underbrace{0}_{1\times n_N}]^{T}$$ 
		with $v$ in the $i-$th position. It can be seen that
		
		$$\ssmall\begin{pmatrix}
			A_1+\beta_{11}H_1C_1&\cdots&\cdots&\cdots&\beta_{1N}H_1C_N\\
			\vdots&\ddots&\cdots&\vdots&\vdots\\
			\beta_{i1}H_2C_1&\cdots&A_i+\beta_{ii}H_iC_i&\cdots&\beta_{iN}H_iC_N\\
			\vdots&\vdots&\vdots&\ddots&\vdots\\
			\beta_{N1}H_NC_1&\ldots&\ldots&\ldots&A_N+\beta_{NN}H_NC_N\\
		\end{pmatrix}\begin{pmatrix}
			0\\\vdots\\v\\\vdots\\0
		\end{pmatrix}$$	\begin{equation*}=\begin{pmatrix}
				0\\\vdots\\\lambda v\\\vdots\\0
			\end{pmatrix}=\lambda V
		\end{equation*}
		
		and
		
		\begin{align*}
			CV&=\begin{pmatrix}
				C_1&0&\cdots&0\\
				0&C_2&\cdots&0\\
				\vdots&\ddots&\vdots\\
				0&\cdots&\cdots&C_N
			\end{pmatrix}\begin{pmatrix}
				0\\\vdots\\v\\\vdots\\0
			\end{pmatrix}=\begin{pmatrix}
				0\\\vdots\\C_iv\\\vdots\\0
			\end{pmatrix}=0
		\end{align*}
		
		Hence by PBH condition, $(\mathcal{C},\mathcal{A})$ is an unobservable pair, as $V\neq 0$.\\
		
		\noindent Conversely suppose that $(C_i,A_i)$ is an observable pair for all $i$, $1\leq i \leq N$.
		Let $(\mathcal{C},\mathcal{A})$ be an unobservable pair.
		Then, by the PBH eigenvector test, there exists an $\mathcal{N}\times 1 $ vector $V\neq 0$ such that 
		\begin{align*}
			\mathcal{A} V=\lambda V\ \text{and}\ \mathcal{C}V=0
		\end{align*} 
		where, $\lambda$ is a scalar.
		Let $V=\begin{bmatrix}
			v_1&v_2&...&v_N
		\end{bmatrix}^T$
		where $v_i\in \mathbb{R}^{n_i},1\leq i \leq N$.
		That is, 
		$$\footnotesize\begin{pmatrix}
			A_1+\beta_{11}H_1C_1&\beta_{12}H_1C_2&\cdots&\beta_{1N}H_1C_N\\
			\beta_{21}H_2C_1&A_2+\beta_{22}H_2C_2&\cdots&\beta_{2N}H_2C_N\\
			\vdots&\vdots&\ddots&\vdots\\
			\beta_{N1}H_NC_1&\beta_{N2}H_NC_2&\cdots&A_N+\beta_{NN}H_NC_N\\
		\end{pmatrix}\begin{pmatrix}
			v_1\\\vdots\\v_k\\\vdots\\v_N
		\end{pmatrix}$$ 
		$$=\lambda\begin{pmatrix}
			v_1\\\vdots\\v_k\\\vdots\\v_N
		\end{pmatrix}$$
		$$\implies \left(A_i+\sum_{j=1}^{N}\beta_{ij}H_iC_j\right)v_i=\lambda v_i; i=1,\cdots, N $$ and
		\begin{align*}
			&\begin{pmatrix}
				C_1&0&\cdots&0\\
				0&C_2&\cdots&0\\
				\vdots&\ddots&\vdots\\
				0&\cdots&\cdots&C_N
			\end{pmatrix}\begin{pmatrix}
				v_1\\\vdots\\v_i\\\vdots\\v_N
			\end{pmatrix}=0\\
			&\implies C_1v_1=C_2v_2=\cdots=C_Nv_N=0
		\end{align*} 
		That is, $A_1v_1=\lambda v_1, A_2v_2=\lambda v_2,\cdots, A_Nv_N=\lambda v_N$. Since $V\neq 0$,  we have $v_i\neq 0$ for some $i$, $1\leq i\leq N$.
		Also we have $C_iv_i=0$. Now, $\mathcal{A} V=\lambda V\implies A_iv_i=\lambda v_i$. This together with $C_iv_i=0$
		implies that $(C_i,A_i)$ is an unobservable pair.
	\end{proof}

	\begin{rem}\label{rem1}
		Theorem \ref{thm4} implies that the observability of the networked system depends only on the observability of the node systems. The network topology $(L,\Delta)$ has no influence in determining the system's observability, i.e., there is no possible way of aligning unobservable nodes, resulting in an observable networked system. By definition, observability refers to the ability of the system to reconstruct the initial state from the knowledge of the output for some finite time interval. If the initial state of the networked system is $X^0=\begin{bmatrix}
			x_1^0,&\cdots &,x_N^0
		\end{bmatrix}$, where $x_i^0$ denotes the initial state of the $i-$th individual node system, $x_i^0$ can be extracted from $X^0$ as follows
		$x_i^0=\begin{bmatrix}
			0&\cdots &I_{n_i}&\cdots 0
		\end{bmatrix}X^0$.
		Similarly, the initial state of the networked system can be reconstructed from that of the individual nodes.
	\end{rem}
	
	\begin{eg}{\label{eg7}}
		To illustrate Theorem \ref{thm4}, we consider a networked system with four nodes. Individual node dynamics are as follows
		\begin{align*}
			&A_1=\begin{pmatrix}
				1&2\\0&1
			\end{pmatrix}, A_2=\begin{pmatrix}
				1&0&1\\1&3&2\\1&0&0
			\end{pmatrix},\\
			&A_3=\begin{pmatrix}
				1&1&2\\1&-1&0\\1&0&1
			\end{pmatrix},
			A_4=\begin{pmatrix}
				1&2\\1&0
			\end{pmatrix}\\
			&B_1=\begin{pmatrix}
				1\\0
			\end{pmatrix}, B_2=\begin{pmatrix}
				1\\2\\-1
			\end{pmatrix}, 
			B_3=\begin{pmatrix}
				1\\0\\1
			\end{pmatrix}, B_4=\begin{pmatrix}
				1\\0
			\end{pmatrix}\\
			&C_1=\begin{pmatrix}
				1&0
			\end{pmatrix},
			C_2=\begin{pmatrix}
				0&0&1
			\end{pmatrix}, C_3=\begin{pmatrix}
				1&0&0
			\end{pmatrix},
			C_4=\begin{pmatrix}
				1&0
			\end{pmatrix}\\
			&H_1=\begin{pmatrix}
				0\\-1
			\end{pmatrix}, H_2=\begin{pmatrix}
				1\\1\\1
			\end{pmatrix},
			H_3=\begin{pmatrix}
				0\\0\\1
			\end{pmatrix}, H_4=\begin{pmatrix}
				1\\1
			\end{pmatrix}
		\end{align*}
		Note that here, nodes $1,3$ and $4$ are observable, wheras node $2$ is unobservable.

		\textbf{(a)} Under the network topology $L=\begin{pmatrix}
			0&0&0&0\\0&0&0&1\\1&0&0&0\\0&0&1&0
		\end{pmatrix}$ and external control inputs applied to nodes $1$ and $4$, the networked system can be written in the compact form \eqref{eq2} with
		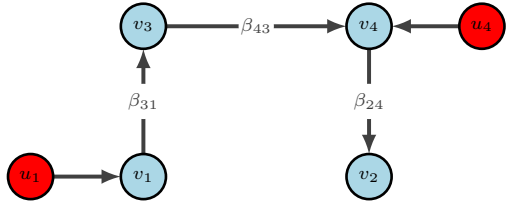
\begin{figure}[h!]\label{fig7a}
			\center
			
			\begin{tikzpicture}
				\Vertex[x=1,y=0,label=$v_1$]{a1}
				\Vertex[x=4,y=2,label=$v_4$]{a4}
				\Vertex[x=1,y=2,label=$v_3$]{a3}
				\Vertex[x=4,y=0,label=$v_2$]{a2}
				\Edge[Direct,label=$\beta_{31}$](a1)(a3)
				\Edge[Direct,label=$\beta_{43}$](a3)(a4)
				\Edge[Direct,label=$\beta_{24}$](a4)(a2)
				\Vertex[x=-0.5,y=0,color=red,label=$u_1$]{b1}
				\Vertex[x=5.5,y=2,color=red,label=$u_4$]{b4}
				\Edge[Direct](b1)(a1)
				\Edge[Direct](b4)(a4)
			\end{tikzpicture} 
			\caption{Unobservable networked system with dynamics as in Example \ref{eg7} (a)}
		\end{figure}
		$$\mathcal{A}=\begin{pmatrix}
			1&2&0&0&0&0&0&0&0&0\\0&1&0&0&0&0&0&0&0&0\\0&0&1&0&1&0&0&0&1&0\\0&0&1&3&2&0&0&0&1&0\\0&0&1&0&0&0&0&0&1&0\\0&0&0&0&0&1&1&2&0&0\\0&0&0&0&0&1&-1&0&0&0\\1&0&0&0&0&1&0&1&0&0\\0&0&0&0&0&1&0&0&1&2\\0&0&0&0&0&1&0&0&1&0
		\end{pmatrix}$$
		$$\mathcal{C}=\begin{pmatrix}
			1&0&0&0&0&0&0&0&0&0\\0&0&0&0&1&0&0&0&0&0\\0&0&0&0&0&1&0&0&0&0\\0&0&0&0&0&0&0&0&1&0
		\end{pmatrix}$$
		By Kalman's rank condition (Theorem \ref{obsv} (i)), we can see that $(C,\mathcal{A})$ is unobservable.
		
		\textbf{(b)} Connecting the nodes in a path network topology, i.e., $L=\begin{pmatrix}
			0&0&0&0\\1&0&0&0\\0&1&0&0\\0&0&1&0
		\end{pmatrix}$
		and external control inputs applied to nodes $1$ and $2$, the networked system can be written in the compact form \eqref{eq2} with
		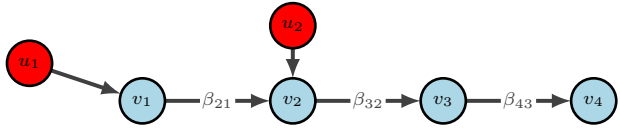
\begin{figure}[h!]\label{fig7b}
			\center
			
			\begin{tikzpicture}
				\Vertex[x=1,y=0,label=$v_1$]{a1}
				\Vertex[x=3,y=0,label=$v_2$]{a2}
				\Vertex[x=5,y=0,label=$v_3$]{a3}
				\Vertex[x=7,y=0,label=$v_4$]{a4}
				\Edge[Direct,label=$\beta_{21}$](a1)(a2)
				\Edge[Direct,label=$\beta_{32}$](a2)(a3)
				\Edge[Direct,label=$\beta_{43}$](a3)(a4)
				\Vertex[x=-0.5,y=0.5,label=$u_1$,color=red]{b1}
				\Vertex[x=3,y=1,label=$u_2$,color=red]{b2}
				\Edge[Direct](b1)(a1)
				\Edge[Direct](b2)(a2)
			\end{tikzpicture} 
			\caption{Unobservable networked system with dynamics as in Example \ref{eg7}(b)}
		\end{figure}
		$$\mathcal{A}=\begin{pmatrix}
			1&2&0&0&0&0&0&0&0&0\\0&1&0&0&0&0&0&0&0&0\\1&0&1&0&1&0&0&0&0&0\\1&0&1&3&2&0&0&0&0&0\\1&0&1&0&0&0&0&0&0&0\\0&0&0&0&0&1&1&2&0&0\\0&0&0&0&0&1&-1&0&0&0\\0&0&0&0&1&1&0&1&0&0\\0&0&0&0&0&1&0&0&1&2\\0&0&0&0&0&1&0&0&1&0
		\end{pmatrix}$$
		and $\mathcal{C}$ is as given above.
		Here too, $(\mathcal{C},\mathcal{A})$ is unobservable.
		This example asserts the fact that network topology has no influence on observability of the network.
	\end{eg}
	
	If the networked system  is homogeneous, then Theorem \ref{thm4} reduces to the result obtained in \cite{A.Thomas} as we see in the following corollary.
	\begin{cor}[Thomas \textit{et al.}, Theorem 4.1, \cite{A.Thomas}]
		The homogeneous networked system \eqref{wang}-\eqref{wang2} is observable if and only if $(C,A)$ is observable.
	\end{cor}
	
	\subsection{Some Necessary conditions for Controllability of the Networked System}
	In this section, we establish some necessary conditions for controllability of the networked system. These results show the relation between controllability and factors like network topology, external control inputs and observability of individual nodes. 
	
	\begin{thm}\label{thm5}\cite{A.Thomas}
		If there exists a node in the network with no in-coming edge, then for the controllability of the networked system \eqref{eq2}-\eqref{eq2-1}, it is necessary that the aforementioned node is controllable under an external control.
	\end{thm}
	
	\begin{proof}
		Suppose there exists a node, say $i$ with no in-coming edges.
		That is, $\beta_{ij}=0,\forall \, j=1,2,...,N$. Then,  the $i$-th block row of $\mathcal{A}$ becomes $$\begin{bmatrix}0&...&0&A_i&0&...&0\end{bmatrix}$$
		Suppose that $(A_i,B_i)$ is not controllable. Then, by the PBH eigenvector test (Theorem \ref{control} (ii)),  there exists a non-zero vector $v\in \mathbb{R}^{n_i}$ and a scalar $\lambda$ such that 
		\begin{align*}
			vA_i=\lambda v,\
			vB_i=0
		\end{align*}
		Then, the $1\times \mathcal{N}$ non-zero vector 
		$$V=[\underbrace{0}_{1\times n_1}\cdots\underbrace{0}_{ 1\times n_{i-1}}\underbrace{v}_{1\times n_i}\  \underbrace{0}_{1\times n_{i+1}}\cdots\underbrace{0}_{1\times n_N}]$$ 
		with $v$ in the $i-$th position will be such that:
		\begin{align*}
			V\mathcal{A}=\lambda V,\
			V\mathcal{B}=0
		\end{align*}
		That is, $(\mathcal{A},\mathcal{B})$ is an uncontrollable pair.\par 
		\noindent Further, suppose that node $i$ does not have external control input, i.e., $\delta_i=0$. Then, $$\begin{bmatrix}0&...&A_i-\lambda I&0&...&0|0&...&0\end{bmatrix}$$ is the $i$-th block row of the matrix $[\mathcal{A}-\lambda I|\mathcal{B}]$. Now, since $rank(A_i-\lambda I)<n$, $rank [\mathcal{A}-\lambda I|\mathcal{B}]<\mathcal{N}$ which implies that system \eqref{eq2}-\eqref{eq2-1} is not controllable.
	\end{proof}

	\begin{thm}\label{thm6}\cite{A.Thomas}
		If there exists a node in the networked system, say $i$ with no external control input, then for the controllability of the networked system \eqref{eq2}-\eqref{eq2-1} it is necessary that $(A_i,H_i)$ is a controllable pair.
	\end{thm}
	\begin{proof}
		Given, node $i$ does not have external control input, i.e., $\delta_i=0$. Suppose that the system \eqref{eq2}-\eqref{eq2-1} is controllable and $(A_i,H_i)$ is an uncontrollable pair. Then, by the PBH eigenvector test (Theorem \ref{control} (ii)), there exists a non-zero vector $v\in \mathbb{R}^{n_i}$ and a scalar $\lambda$ such that 
		\begin{align*}
			vA_i=\lambda v,\
			vH_i=0
		\end{align*} 
		Then the vector $V=[\underbrace{0}_{1\times n_1}\cdots\underbrace{0}_{1\times n_{i-1}}\ \underbrace{v}_{1\times n_i}\ \underbrace{0}_{1\times n_{i+1}}\cdots\ \underbrace{0}_{1\times n_N}]$ with $v$ in the $i-$th position will be such that:
		\begin{align*}V\mathcal{A}=\lambda V,\ V\mathcal{B}=0\end{align*}
		That is, system \eqref{eq2}-\eqref{eq2-1}  is uncontrollable. 
	\end{proof}

	\begin{thm}\label{thm7}\cite{A.Thomas}
		Let $\Delta^{+}$ denote the set of all nodes in the networked system \eqref{eq2}-\eqref{eq2-1} with external control inputs. If $N>\sum_{i\in \Delta^{+}}rank(B_i)$, then for the controllability of \eqref{eq2}-\eqref{eq2-1}, it is necessary that $(C_j,A_j+\beta_{jj}H_jC_j)$ is an observable pair for some $j, 1\leq j\leq N$.
	\end{thm}
	
	\begin{proof}
		Given, $N>\sum_{i\in \Delta^{+}}rank(B_i)$. Suppose $(\mathcal{A},\mathcal{B})$ is a controllable pair and let $(C_j,A_j+\beta_{jj}H_jC_j)$ be unobservable $\forall\, j, 1\leq j \leq N$. $$[\beta_{1j}H_1C_j\,...\,A_j+\beta_{jj}H_jC_j-\lambda I\,...\,\beta_{Nj}H_NC_j]^T$$ is the $j$-th column block of the matrix $[\mathcal{A}-\lambda I|\mathcal{B}]$. The unobservability of $(C_j,A_j+\beta_{jj}H_jC_j)$ implies that there exists an $n_j\times 1$ vector $v(\neq 0)$ and a scalar $\lambda$ such that
		\begin{align*}(A_j+\beta_{jj}H_jC_j)v=\lambda v\ \text{and}\ C_jv=0\end{align*}
		Hence rank of the $j$-th column block $[\beta_{1j}H_1C_j...A_j+\beta_{jj}H_jC_j-\lambda I...\beta_{Nj}H_NC_j]^T$ reduces atleast by one, and hence $$\ rank(\mathcal{A}-\lambda I)\leq \mathcal{N}-N$$
		Further, since $N>\sum_{i\in \Delta^{+}}B_i$, $rank[\mathcal{A}-\lambda I|\mathcal{B}]<\mathcal{N}$, hence $(\mathcal{A},\mathcal{B})$ is an uncontrollable pair.
	\end{proof}

	For the homogeneous networked system \eqref{wang}-\eqref{wang2}, the necessary conditions obtained by Wang \textit{et al.} in \cite{L.Wang} can be derived as corollaries of the theorems in this section.
	\begin{cor}[Wang \textit{et al.}, Theorem 2, \cite{L.Wang}]
		If there exists one node without incoming edges, then to
		reach controllability of the homogeneous networked system , it is necessary
		that $(A,B)$ is a controllable pair and moreover an external control input is
		applied onto this node which has no incoming edges.
	\end{cor}
	
	\begin{cor}[Wang \textit{et al.}, Theorem 3, \cite{L.Wang}]
		If there exists one node without external control inputs,
		then for the homogeneous networked system to be controllable, it is necessary
		that $(A,HC)$ is a controllable pair.
	\end{cor}
	
	\begin{cor}[Wang \textit{et al.}, Theorem 4, \cite{L.Wang}]
		If the number of nodes with external control inputs is $m$,
		and $N>m\cdot rank(B)$, then for the homogeneous networked system to be
		controllable, it is necessary that $(A,C)$ is an observable pair.
	\end{cor}

	\section{Conclusion}\label{con}
	In this work we have studied the controllability and observability of a networked system $(\mathcal{A},\mathcal{B})$ with individual node dynamics given by $(A_i, B_i, C_i)$. The node dimensions as well as the inner-coupling matrices, which describe the inner-interactions among the nodes, are distinct. The nodes are connected in the network topology $(L,\Delta)$. We have obtained two necessary and sufficient  conditions for controllability and a necessary and sufficient  condition for observability of the network. The second necessary and sufficient  condition for controllability is illustrative in nature as it brings forth the relation between controllability and node-specific factors like nodal dynamics, inner-interactions among nodes as well as the network topology. The condition for observability states that the underlying graph topology has no influence in determining observability of the network. Further, the results state that under certain conditions on the network topology and external control inputs, the existence of a controllable node, the existence of a controllable pair $(A_i,H_i)$ and the existence of an observable node are necessary for controllability of the network, nevertheless these conditions are not sufficient. In future works, we envisage to formulate further efficient necessary and sufficient  conditions for controllability of heterogeneous networks with distinct node dimensions.

	\appendix
	\section{Specific Network Topologies}
	Here, we consider a few cases of network topologies, namely path, cycle, star and wheel where controllability conditions are easily verifiable.\\
	
	\textbf{(i) Path Network Topology} 
	Consider the path network $P_N$ with $N$ nodes. The networked system $(\mathcal{A},\mathcal{B})$ under $P_N$ is:
	\begin{align*}
		\mathcal{A}&=\begin{pmatrix}
			A_1&0&\cdots&0\\
			\beta_{21}H_2C_1&A_2&\cdots&0\\
			\vdots&\ddots&\quad&\vdots\\
			0&\cdots&\beta_{nn-1}H_nC_{n-1}&A_N
		\end{pmatrix}\\
		\mathcal{B}&=diag.\{\delta_1B_1,\delta_2B_2,...,\delta_NB_N\}
	\end{align*}
	\begin{figure}[h!]\label{fig12}
		\centering
		\begin{tikzpicture}
			\Vertex[label=$v_1$]{a1};
			\Vertex[x=2,label=$v_2$]{a2};
			\Vertex[x=4,label=$v_3$]{a3};
			\Edge[Direct](a1)(a2);
			\Edge[Direct](a2)(a3);
		\end{tikzpicture} 
		\caption{$P_3$, the path graph on $3$ nodes}
	\end{figure}
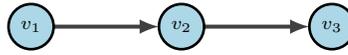
	
	\begin{cor}\cite{A.Thomas}
		For a networked system $(\mathcal{A},\mathcal{B})$ wherein the nodes are connected in path topology to be controllable, it is necessary that $(A_1,B_1)$ is a controllable pair.
	\end{cor}
	
	\begin{proof}
		For nodes aligned under path topology, the first node does not have an incoming edge. Also, there are no self-loops, i.e., $\beta_{ii}=0$. The corollary follows directly from Theorem \ref{thm5}.
	\end{proof}
	\textbf{(ii) Cycle Network Topology} 
	Consider an $N-$cycle $C_N$. Here, $$\mathcal{A}=\begin{pmatrix}
		A_1&0&\cdots&\beta_{1N}H_1C_N\\
		\beta_{21}H_2C_1&A_2&\cdots&0\\
		0&\beta_{32}H_3C_2&A_3&0\\
		\vdots&\ddots&\cdots&0\\
		0&\cdots&\beta_{NN-1}H_NC_{N-1}&A_N
	\end{pmatrix}$$ and $\mathcal{B}=diag\{\delta_1B_1,\cdots,\delta_NB_N\}$.
	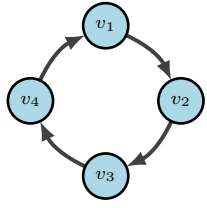
\begin{figure}[h!]\label{fig11}
		\centering
		\begin{tikzpicture}
			\Vertex[x=0,y=1,label=$v_1$]{a1}
			\Vertex[x=1,y=0,label=$v_2$]{a2}
			\Vertex[x=0,y=-1,label=$v_3$]{a3}
			\Vertex[x=-1,y=0,label=$v_4$]{a4}
			\Edge[Direct, bend=20](a1)(a2)
			\Edge[Direct, bend=20](a2)(a3)
			\Edge[Direct, bend=20](a3)(a4)
			\Edge[Direct, bend=20](a4)(a1)
		\end{tikzpicture} 
		\caption{$C_4$, the cycle graph on $4$ nodes}
	\end{figure}
	
	\begin{cor}
		Consider an $N-$ node networked system under cycle topology. Let $\sigma(A_{i})=\{ \mu_{i}^{1}, \mu_{i}^{2},\cdots,\mu_{i}^{q_i}\}, 1\leq i\leq N$ and $\zeta_{ij}^k$ denote the left eigenvectors of $A_{i}$ corresponding to the eigenvalue $\mu_i^j$ where $1\leq k\leq \gamma_{ij}$, $\gamma_{ij}$ denotes the geometric multiplicity of $\mu_i^j$ as an eigenvalue of $A_{i}$. Suppose $H_{j+1}C_{j}=0, j=1,\cdots, N$ and $\zeta_{1N}^k$ is orthogonal to $H_{1}C_{N}$, $k=1,\cdots,\gamma_{ij}$. The system is controllable if and only if:
		\begin{itemize}
			\item[(i)] $(A_i,B_i)$ is a controllable pair $\forall i, 1\leq i\leq N$
			\item [(ii)] Each node is under external control
		\end{itemize} 
		
		\begin{proof}
			In cycle topology, $\beta_{ii}=0,i=1,2,\ldots,N$. Hence, the corollary follows from Theorem \ref{thm3}.
		\end{proof}
		
	\end{cor}
	
	\textbf{(iii) Star Network Topology} 
	Consider the star network $S_N$ with $N$ nodes and with node $l$ as the star, i.e., the node of degree $N-1$. Then the networked system $(\mathcal{A},\mathcal{B})$ under $S_N$ is:
	
	\begin{align*}
		\mathcal{A}&=\begin{pmatrix}
			A_1&0&\cdots&\beta_{1l}H_1C_l&\cdots&0\\
			0&A_2&\cdots&\beta_{2l}H_2C_l&\cdots&0\\
			\vdots&\quad&\vdots&\vdots&\vdots\\
			0&\cdots&0&\beta_{Nl}H_NC_l&\cdots&A_N
		\end{pmatrix}\\
		\mathcal{B}&=diag\{\delta_1B_1,\cdots,\delta_NB_N\}
	\end{align*}
	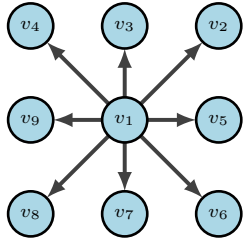
\begin{figure}[h!]\label{fig12}
		\centering
		\begin{tikzpicture}
			\Vertex[x=0,y=0,label=$v_1$]{a1}
			\Vertex[x=1.25,y=1.25,label=$v_2$]{a2}
			\Vertex[x=0,y=1.25,label=$v_3$]{a3}
			\Vertex[x=-1.25,y=1.25,label=$v_4$]{a4}
			\Vertex[x=1.25,y=0,label=$v_5$]{a5}
			\Vertex[x=1.25,y=-1.25,label=$v_6$]{a6}
			\Vertex[x=0,y=-1.25,label=$v_7$]{a7}
			\Vertex[x=-1.25,y=-1.25,label=$v_8$]{a8}
			\Vertex[x=-1.25,y=0,label=$v_9$]{a9}
			\Edge[Direct](a1)(a2)
			\Edge[Direct](a1)(a3)
			\Edge[Direct](a1)(a4)
			\Edge[Direct](a1)(a5)
			\Edge[Direct](a1)(a6)
			\Edge[Direct](a1)(a7)
			\Edge[Direct](a1)(a8)
			\Edge[Direct](a1)(a9)
		\end{tikzpicture} 
		\caption{$S_9$, the star graph on $9$ nodes}
	\end{figure}
	
	\begin{cor}\label{cor-st1}
		Consider an $N-$ node networked system under star topology with node $l(1\leq l\leq N)$ as the star. Let $\sigma(A_{i})=\{ \mu_{i}^{1}, \mu_{i}^{2},\cdots,\mu_{i}^{q_i}\}, 1\leq i\leq N$ and $\zeta_{ij}^k$ denote the left eigenvectors of $A_{i}$ corresponding to the eigenvalue $\mu_i^j$ where $1\leq k\leq \gamma_{ij}$, $\gamma_{ij}$ denotes the geometric multiplicity of $\mu_i^j$ as an eigenvalue of $A_{i}$. Suppose $H_{i}C_{l}=0, \forall i>l$ and $\zeta_{il}^k$ is orthogonal to $H_{i}C_{l}, \forall i<l$, $k=1,\cdots,\gamma_{ij}$. The system is controllable if and only if:
		\begin{itemize}
			\item[(i)] $(A_i,B_i)$ is a controllable pair $\forall i, 1\leq i\leq N$
			\item [(ii)] each node is under external control
		\end{itemize} 
	\end{cor}
	
	\begin{proof}
		In star topology, $\beta_{ii}=0,i=1,2,\ldots,N$. Hence, the corollary follows from Theorem \ref{thm3}.
	\end{proof}
	
	\textbf{(iv) Wheel Network Topology} 
	Consider the wheel network $W_N$ with $N$ nodes and without loss of generality let node $1$ be the hub, i.e., the star node of $S_{N}$ underlying $W_N$. Then the networked system $(\mathcal{A},\mathcal{B})$ under $W_N$ is:
	\begin{equation*}
		\scriptstyle\mathcal{A}=\begin{pmatrix}
			A_1&0&\cdots&\cdots&\beta_{1N}H_1C_N\\
			\beta_{21}H_2C_1&A_2&\cdots&\cdots&0\\
			\vdots &\ddots& \ddots&\cdots&0\\
			\vdots&\vdots&\ddots& \ldots &\vdots \\
			\beta_{N1}H_NC_1 &0 &\ldots& \beta_{N(N-1)}H_NC_{N-1}&A_N
	\end{pmatrix}\end{equation*}
	\begin{equation*}	\mathcal{B}=diag.\{\delta_1B_1,\delta_2B_2,...,\delta_NB_N\}
	\end{equation*}
	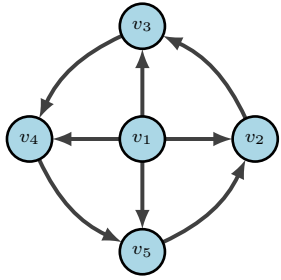
\begin{figure}[h!]\label{fig13}
		\centering
		\begin{tikzpicture}
			\Vertex[x=0,y=0,label=$v_1$]{a1};
			\Vertex[x=1.5,y=0,label=$v_2$]{a2};
			\Vertex[x=0,y=1.5,label=$v_3$]{a3};
			\Vertex[x=-1.5,y=0,label=$v_4$]{a4};
			\Vertex[x=0,y=-1.5,label=$v_5$]{a5};
			\Edge[Direct](a1)(a2);
			\Edge[Direct](a1)(a3);
			\Edge[Direct](a1)(a4);
			\Edge[Direct](a1)(a5);
			\Edge[bend=-20,Direct](a2)(a3);
			\Edge[bend=-20,Direct](a3)(a4);
			\Edge[bend=-20,Direct](a4)(a5);
			\Edge[bend=-20,Direct](a5)(a2);
		\end{tikzpicture} 
		\caption{$W_5$, the wheel graph on $5$ nodes}
	\end{figure}
	
	\begin{cor}
		Consider an $N-$ node networked system under wheel topology with node $l(1\leq l\leq N)$ as the hub. Let $\sigma(A_{i})=\{ \mu_{i}^{1}, \mu_{i}^{2},\cdots,\mu_{i}^{q_i}\}, 1\leq i\leq N$ and $\zeta_{ij}^k$ denote the left eigenvectors of $A_{i}$ corresponding to the eigenvalue $\mu_i^j$ where $1\leq k\leq \gamma_{ij}$, $\gamma_{ij}$ denotes the geometric multiplicity of $\mu_i^j$ as an eigenvalue of $A_{i}$. Suppose $H_iC_1=0, i=3,\cdots, N$, $H_{j+1}C_j=0, j=1,\cdots,N$ and $v_1\perp H_1C_N$. Then the system is controllable if and only if:
		\begin{itemize}
			\item[(i)] $(A_i,B_i)$ is a controllable pair $\forall i, 1\leq i\leq N$
			\item [(ii)] each node is under external control
		\end{itemize} 
	\end{cor}
	
	\begin{proof}
		In wheel topology, $\beta_{ii}=0,i=1,2,\ldots,N$. Hence, the corollary follows from Theorem \ref{thm3}.
	\end{proof}

	\section*{Acknowledgment}
	The first and second authors thank the University Grants Commission, India and the Council for Scientific and Industrial Research, India for the financial support, respectively. We thank the Department of Mathematics, Indian Institue of Space Science and Technology, Thiruvananthapuram, Kerala, India for providing all the necessary facilities to pursue this research work.

\end{document}